\documentclass[a4paper]{article}
\usepackage[utf8]{inputenc}
\usepackage{amsthm,amssymb,amsmath,enumerate,graphicx}
\usepackage[usenames,dvipsnames]{color}
\usepackage{tikz}
\usepackage[british]{babel}
\usepackage{caption}
\usepackage{subcaption}
\usetikzlibrary{positioning,fit,shapes,calc}
\usetikzlibrary{decorations.pathmorphing,decorations.shapes}

\newcommand{\m}{m}

\usetikzlibrary{calc,shadings}

\usepackage{pgfplots}

\newenvironment{customlegend}[1][]{%
	\begingroup
	\csname pgfplots@init@cleared@structures\endcsname
	\pgfplotsset{#1}%
}{%
\csname pgfplots@createlegend\endcsname
\endgroup
}%

\def\addlegendimage{\csname pgfplots@addlegendimage\endcsname}

\pgfkeys{/pgfplots/number in legend/.style={%
		/pgfplots/legend image code/.code={%
			\node at (0.295,-0.0225){#1};
		},%
	},
}

\makeatletter
\def\@cite#1#2{{\normalfont[{\bfseries#1\if@tempswa , #2\fi}]}}
\makeatother

\let\eps\varepsilon 
\renewcommand{\epsilon}{\varepsilon}

\renewcommand{\subset}{\subseteq}

\theoremstyle{plain}

\newtheorem{theorem}{Theorem}[section]
\newtheorem{corollary}[theorem]{Corollary}
\newtheorem{lemma}[theorem]{Lemma}
\newtheorem{fact}[theorem]{Fact}
\newtheorem{claim}[theorem]{Claim}
\newtheorem{conjecture}[theorem]{Conjecture}

\definecolor{myred}{RGB}{0,0,0}
\definecolor{mygreen}{RGB}{0,0,0}
\definecolor{myblue}{RGB}{0,0,0}

\newcommand{\kn}{K_{n,n}}

\newcommand{\bitop}[1]{\overline{{#1}}}
\newcommand{\bibot}[1]{\underline{{#1}}}

\title{Local colourings and monochromatic partitions in complete bipartite graphs}
\author{\bigskip Richard Lang and
	Maya Stein\thanks{The second author was supported by Fondecyt Regular no.~1140766, and both authors acknowledge support by the Millenium Nucleus Information and Coordination in Networks.}\\Universidad de Chile\\ 
	Santiago, Chile\\
	\texttt{rlang@dim.uchile.cl, mstein@dim.uchile.cl}}

\begin{document}

\maketitle

\begin{abstract}
We show that  for any $2$-local colouring of the edges of the balanced complete bipartite graph $K_{n,n}$, its vertices can be covered with at most~$3$ disjoint monochromatic paths. And, we can cover almost all vertices of any  complete or balanced complete bipartite $r$-locally
coloured graph with $O(r^2)$  disjoint monochromatic cycles.\\ 
We also determine the $2$-local bipartite Ramsey number of a path almost exactly: Every $2$-local colouring of the edges of $K_{n,n}$ contains a monochromatic path on $n$ vertices.
\end{abstract}

\section{Introduction}
The problem of partitioning a graph into few monochromatic paths or cycles, first formulated explicitly in the beginning of the 80's~\cite{Gya83}, 
has lately received a fair amount of attention. Its origin lies in Ramsey theory and its subject are complete graphs (later
substituted with other types of graphs), whose edges are coloured with $r$ colours. Call such a colouring an $r$-colouring; 
note that
this need not be a proper edge-colouring. The challenge is now to find a small number of disjoint monochromatic paths, which 
together cover the vertex set of the underlying graph. Or, instead of disjoint monochromatic paths, we might ask for disjoint monochromatic cycles. 
Here, single vertices and edges count as cycles. Such a cover is called a monochromatic path partition, or a monochromatic cycle partition, respectively. 
It is not difficult to construct $r$-colourings that do not allow for partitions into less than $r$ paths, or cycles.\footnote{For instance, 
	take vertex sets $V_1,\ldots, V_r$ with $|V_i|=2^i$, and for $i\leq j$ give all $V_i$--$V_j$ edges colour $i$.}

At first, the problem was studied mostly for $r=2$, and the complete graph~$K_n$ as the host graph. In this situation, a partition into two disjoint paths always exists~\cite{GG67}, 
regardless of the size of $n$. Moreover, we can require these paths to have different colours. An extension of this fact, namely that every $2$-colouring of $K_n$ has a partition into two monochromatic 
cycles of different colours was conjectured by Lehel, and verified by Bessy and Thomass\'e~\cite{BT10}, after 
preliminary work for large $n$~\cite{All08,LRS98}.

A generalisation of these two results for other values of $r$, i.e.~that any $r$-coloured $K_n$ can be partitioned into $r$ monochromatic paths, or  into $r$ monochromatic cycles,  was conjectured by Gy\'arf\'as~\cite{Gya89} and by Erd\H os, Gy\'arf\'as and Pyber~\cite{EGP91}, respectively. The conjecture for cycles was recently disproved by Pokrovskiy~\cite{Pok14}. He gave counterexamples for all $r\geq 3$, but he also showed that the conjecture for paths is true for $r=3$.
Gy\'arf\'as, Ruszink\'o, S\'ark\"ozy and Szemer\'edi~\cite{GRSS06} showed that any $r$-coloured $K_n$ can be partitioned into $O(r\log r)$ monochromatic cycles, improving an earlier bound from~\cite{EGP91}.

Monochromatic path/cycle partitions have also been studied for bipartite graphs, mainly for $r=2$.
A $2$-colouring of $K_{n,n}$ is called
a split colouring if
there is a colour-preserving homomorphism from the edge-coloured $K_{n,n}$ to a properly edge-coloured $K_{2,2}$. Note that any split colouring allows for a partition into three paths, but not always into two. However, split colourings are the only `problematic' colourings, as the following result shows.

\begin{theorem}[Pokrovskiy~\cite{Pok14}]
\label{thm:pokrovskiy}
Let the edges of $\kn$ be coloured with 2 colours; then $\kn$ can be partitioned into two paths of distinct colours or the colouring is split. 
\end{theorem}

Split colourings can be generalised to more colours~\cite{Pok14}. This gives  a lower bound of $2r-1$ on the path/cycle partition number for $K_{n,n}$. For $r=3$, this bound is asymptotically correct~\cite{LSS15}.
For an upper bound, Peng, R\"odl and Ruci\'nski~\cite{PRR02} showed that any $r$-coloured $K_{n,n}$ can be partitioned into $O(r^2 \log r)$ monochromatic cycles, improving a result of Haxell~\cite{Hax97}. We improve this bound to $O(r^2)$.

\begin{theorem}\label{thm:bip}
For every $r \geq 1$ there is  an $n_0$ such that for $n \geq n_0$, for any $r$-locally coloured  $K_{n,n}$, we need at most $4 r^2$ disjoint monochromatic cycles to cover all its vertices.
\end{theorem}

Theorem~\ref{thm:bip} follows immediately from Theorem~\ref{thm:complete} (b) below.
Let us mention that the monochromatic cycle partition problem has also been studied for multipartite graphs~\cite{SS14}, and for 
arbitrary graphs~\cite{BBG+14,Sar11}, or replacing paths or cycles with other graphs~\cite{GSS+12,SS00,SSS12b}.

\bigskip

Our main focus in this paper is on monochromatic cycle partitions for {\it local colourings} (Theorem~\ref{thm:bip} being only a side-product of our local colouring results). Local colourings are a natural way to generalise $r$-colourings.
A colouring is $r$-local if no vertex is adjacent to more than $r$ edges of distinct colours. Local colorings have appeared mostly in the context of Ramsey theory~\cite{Bie03,CT93,GLN+87,GLS+87,RT97,Sch97,Tru92,TT87}. 

With respect to monochromatic path or cycle partitions,
Conlon and Stein~\cite{CS16} recently generalised some of the above mentioned results to $r$-local colourings. They show that for any $r$-local colouring of $K_n$, there is a partition into $O(r^2 \log r)$ monochromatic cycles, and, if $r=2$, then two cycles suffice. 
In this paper we  improve their general bound for complete graphs, and give the first bound for  monochromatic cycle partitions in bipartite graphs.  In both cases, $O(r^2)$ cycles suffice. 

\begin{theorem}\label{thm:complete}
For every $r \geq 1$ there is  an $n_0$ such that for $n \geq n_0$ the following holds. 
\begin{enumerate}[(a)]
\item  If $K_{n}$ is $r$-locally coloured, then its vertices can be covered with at most $2 r^2$ disjoint monochromatic cycles.
\item If $K_{n,n}$ is $r$-locally coloured, then its vertices can be covered with at most~$4 r^2$ disjoint monochromatic cycles.
\end{enumerate}
\end{theorem}
We do not believe our results are best possible, but suspect  that in both cases ($K_n$ and $K_{n,n}$), the number of cycles needed should be linear in $r$.

\begin{conjecture}
There is a $c$ such that for every $r$, every $r$-local colouring of $K_n$ admits a covering with $cr$ disjoint cycles. The same should hold replacing $K_n$ with $K_{n,n}$.
\end{conjecture}

\smallskip

Our second result is a generalisation of Theorem~\ref{thm:pokrovskiy} to local colourings:
\begin{theorem}
\label{thm:bipartite-2-local-paths-partition}
Let the edges of $\kn$ be coloured 2-locally. Then $\kn$ can be partitioned into 3 or less monochromatic paths.
\end{theorem}
So, in terms of monochromatic path partitions, it does not matter whether our graph is 2-locally coloured, or if the total number of colours is 2. For more colours this might be different, but we have not been able to construct $r$-local colourings of $K_{n,n}$ which need more than $2r-1$ monochromatic paths for covering the vertices.

We prove Theorem~\ref{thm:complete} in Section~\ref{sec:2} and  Theorem~\ref{thm:bipartite-2-local-paths-partition} in Section~\ref{sec:2-local-path-partition}. These proofs are totally independent of each other.

\medskip

Theorem~\ref{thm:bipartite-2-local-paths-partition} relies on a structural lemma for $2$-local colourings, Lemma~\ref{3possiblecolourings}. This lemma has a second application in local Ramsey theory.
As mentioned above, some effort has gone into extending Ramsey theory to local 
colourings. 
In particular, in~\cite{GLS+87}, Gy\'arf\'as et al.~determine the $2$-local Ramsey number of the path $P_n$. This number is defined as the smallest number  $m$ such that in any $2$-local colouring of $K_m$ a monochromatic path of length $n$ is present. In~\cite{GLS+87}, it is shown that the $2$-local Ramsey number of the path $P_n$ is $\lceil\frac 32n -1\rceil$. Thus the usual $2$-colour Ramsey number of the path, which is $\lfloor\frac 32n -1\rfloor$ and the $2$-local Ramsey number of the path $P_n$ only differ by at most 1 (depending on the parity of $n$).

The {\it bipartite} $2$-colour Ramsey number of the path $P_n$ is defined as a pair $(m_1,m_2)$, with $m_1\geq m_2$ such that for any pair $m_1',m_2'$ we have that $m'_i\geq m_i$ for both $i=1,2$ if and only if every $2$-colouring of $K_{m_1',m_2'}$ contains a monochromatic path $P_{n}$. Gy\'arf\'as and Lehel~\cite{GL73} and, independently, Faudree and Schelp~\cite{FS75} determined the bipartite $2$-colour Ramsey number of  $P_{2m}$ to be $(2m-1,2m-1)$. The authors of~\cite{FS75} also show that for the odd path $P_{2m+1}$ this number is $(2m+1,2m-1)$. Observe that suitable split colourings can be used to see the sharpness of these Ramsey numbers.

We use our auxiliary structural result, Lemma~\ref{3possiblecolourings}, and the result of~\cite{GL73} to determine the $2$-local bipartite Ramsey number for the even path $P_{2m}$. As for complete host graphs, it turns out this number coincides with its non-local pendant.

\begin{theorem}\label{thm:ramsey}
Let $K_{2\m-1,2\m-1}$ be coloured 2-locally. Then there is a monochromatic path on $2\m$ vertices.
\end{theorem}

It is likely that similar arguments can be applied to obtain an analoguous result for odd paths (but such an analogue is not straightforward). Clearly, the result from~\cite{FS75} together with Theorem~\ref{thm:ramsey} (for $m+1$) imply that the $2$-local bipartite Ramsey number for the odd path $P_{2m+1}$ is one of $(2m+1,2m-1)$, $(2m+1,2m)$, $(2m+1,2m+1)$.

\medskip

In view of the results from~\cite{CS16} and our Theorems~\ref{thm:complete},~\ref{thm:bipartite-2-local-paths-partition} and~\ref{thm:ramsey}, it might seem that in terms of path- or cycle-partitions, $r$-local colourings are not very different from $r$-colourings. Let us give an example where they do behave differently, even for $r=2$. 

It is shown in~\cite{SS14} that any 2-coloured complete tripartite graph can be partitioned into at most 2 monochromatic paths, provided that no part of the tripartition contains more than half of the vertices.
This is not true for 2-local colourings:
Let $G$ be a complete tripartite graph with triparts $U$, $V$ and $W$ such that $|U| = 2 |V| = 2 |W| \geq 6$. Pick vertices $u \in U$, $v \in V$ and $w \in W$ and write
$U' = U\setminus \{u\}$, $V' = V\setminus \{v\}$ and $W' = W\setminus \{w\}$.
Now colour the edges of $[W' \cup \{v\}, U']$ red, 
$[V' \cup \{w\}, U']$ green and the remaining edges blue. 
This is a 2-local colouring. However, since no monochromatic path can cover all vertices of $U'$, we need at least 3 monochromatic paths to cover all of $G$.

Note that in our example, the graph $G$ contains a 2-locally coloured balanced complete bipartite graph. This shows that in the situation of  Theorem~\ref{thm:bipartite-2-local-paths-partition}, we might need 3 paths even if the 2-local colouring is not a split colouring (and thus a 2-colouring). Blowing this example up, and adding some smaller sets of vertices seeing new colours, one obtains examples of $r$-local colourings of balanced complete bipartite graphs requiring $2r-1$ monochromatic paths.

\section{Proof of Theorem~\ref{thm:complete}}
\label{sec:2}

In this section we will prove our bounds for monochromatic cycle partitions, 
given by Theorem~\ref{thm:complete}.
The heart of this section is Lemma~\ref{lem:local-matchings}. This lemma enables us to use induction on $r$, in order  to prove new bounds for the number of monochromatic matchings 
needed to cover an $r$-locally coloured graph. In particular, we find these bounds for the complete and the complete bipartite graph. All of this is the topic of Subsection~\ref{match}.

To get from monochromatic cycles to the promised cycle cover,
we use a nowadays standard approach, which was first introduced in~\cite{Luc99}. 
We find a large robust hamiltonian graph, regularise the rest, find monochromatic matchings covering almost all, 
blow them up to cycles, and the absorb the remainder with the robust hamiltonian graph. The interested reader may find a sketch of this well-known method in Subsection~\ref{fromto}.

\subsection{Monochromatic matchings}\label{match}

Given a graph $G$ with  an edge colouring, a monochromatic connected matching is a matching in a connected component of the subgraph that is induced by the edges of a single colour.

\begin{lemma}
\label{lem:local-matchings}
For $k \geq 2$, let the edges of a graph $G$ be coloured $k$-locally. Suppose there are $m$ monochromatic components that together cover $V(G)$, of colours $c_1,\ldots,c_m$. \\
Then there are $m$ 
vertex-disjoint monochromatic connected matchings $M_1,\ldots,$ $M_m$, of colours $c_1,\ldots,c_m$, such that the inherited colouring of $G\setminus V(\bigcup_{i=1}^{m}M_i)$ is a $(k-1)$-local colouring.
\end{lemma}
\begin{proof}
Let $G$ be covered by $m$ monochromatic components $C_1, \ldots, C_m$ of colours  $c_1, \ldots, c_m$. Let $M_1 \subset C_1$ be a  maximum matching
in colour $c_1$. For $ 2 \leq i \leq m$ we iteratively pick maximum matchings $M_i \subset C_i \setminus V(\bigcup_{j<i}M_j) $ in colour $c_i$. 
Set $M:= \bigcup_{j\leq m}M_j$.

Now let $v$ be any vertex in $H:= G \setminus V(M)$. Say $v \in V(C_i \setminus V(M)).$ In particular,  vertex $v$ sees colour $c_i$ in $G$. However, by maximality of $M_i$, vertex $v$ does not see the colour $c_i$ in $H$. Thus in $H$, vertex $v$ sees at most $k-1$ colours. Hence,  the inherited colouring of $H$ is a $(k-1)$-local colouring, which is as desired. 
\end{proof}

\begin{corollary}\label{cor:robust-cycle-complete}
If $K_n$ is $r$-locally edge coloured, and $H$ is obtained from $K_n$ by deleting  $o(n^2)$ edges, then 
\begin{enumerate}[(a)]
\item $V(K_n)$ can be covered with at most $r(r+1)/2$ monochromatic connected matchings, and 
\item all but $o(n)$ vertices of $H$ can be covered with at most $r(r+1)/2$ monochromatic connected matchings.
\end{enumerate}
Note that the matchings from (b) are connected in $H$.
\end{corollary}
\begin{proof}
The proof is based on the following easy observation. In any colouring of $K_n$, 
the closed monochromatic neighbourhoods of any vertex $v$ together cover $K_n$. Since the colouring is a $k$-local colouring, we can cover 
all of $V(K_n)$ with $k$ components. Now apply Lemma~\ref{lem:local-matchings} successively to obtain the bound from (a).

For (b), it suffices to observe that  we can choose at each step a vertex $v$ that has $o(n)$ non-neighbours in the current graph. 
For, if at some step, there is no such vertex, then a simple calculation shows we have already covered all but $o(n)$ of $V(K_n)$, and can hence abort the procedure.
\end{proof}

\begin{corollary}\label{cor:robust-cycle-bipartite}
If $K_{n,n}$ is $r$-locally edge coloured, and $H$ is obtained from $K_{n,n}$ by deleting  $o(n^2)$ edges, then 
\begin{enumerate}[(a)]
\item $V(K_{n,n})$ can be covered with at most $(2r-1)r$ monochromatic connected matchings, and 
\item all but $o(n)$ vertices of $H$ can be covered with at most $(2r-1)r$ monochromatic connected matchings.
\end{enumerate}
Note that the matchings from (b) are connected in $H$.
\end{corollary}
\begin{proof}
The proof very similar to the proof Corollary~\ref{cor:robust-cycle-complete}. 
We only note that in any colouring of $K_{n,n}$ the two closed monochromatic
neighbourhoods of any edge form a vertex cover of size at most $2r-1$. 
\end{proof}

\subsection{From matchings to cycles}\label{fromto}
\subsubsection{Regularity}
Regularity is the key for expanding our partition  of an $r$-locally coloured $K_n$ or $K_{n,n}$ into monochromatic connected matchings to a partition of almost all vertices into monochromatic cycles.
We follow 
an approach introduced by \L uczak~\cite{Luc99}, which has become a standard method for cycle embeddings in large graphs. We will focus on the parts where our proof differs
from  other applications of this method (see~\cite{GRSS06,GRSS11,LSS15}).

The main result of this section is:

\begin{lemma}\label{lem:asymptotic-cycle-partition}
If $K_n$ and $K_{n,n}$ are $r$-locally edge coloured, then
\begin{enumerate}[(a)]
\item all but $o(n)$ vertices of $K_n$ can be covered with at most $r(r+1)/2$ monochromatic cycles.
\item all but $o(n)$ vertices of $\kn$ can be covered with at most $(2r-1)r$ monochromatic cycles.
\end{enumerate}
\end{lemma}

Before we start, we need a couple of regularity preliminaries. 
For a graph $G$ and disjoint subsets of vertices $A,B \subset V(G)$ we denote by $[A,B]$ the bipartite subgraph with biparts $A$ and $B$ and edge set $\{ab \in E(G): a\in A, b \in B\}$.
We write $\deg_G(A,B)$ for the number of edges in $[A,B]$. If $A=\{a\}$ we write shorthand $\deg_G(a,B)$.

The subgraph $[A,B]$ is \emph{$(\eps, G)$-regular} if
$$ |\deg_G(X,Y) - \deg_G(A,B)| < \eps$$ for all  $X \subset A,~Y\subset B$ with $|X|> \eps |A|,~|Y| > \eps|B|$.
Moreover, ${[A,B]}$ is \emph{$(\eps,\delta,G)$-super-regular} if it is $(\eps,G)$-regular and 
$$ \deg_{G}(a,B) > \delta |B|\text{ for each } a \in A \text{ and }\deg_{G}(b,A) > \delta |A| \text{ for each } b \in B .$$
A vertex-partition $ \{ V_0, V_1 , \ldots , V_l \}$ of the vertex set of a graph $G$  into $l+1$ \emph{clusters} is called \emph{$(\eps,G)$-regular}, if 
\begin{enumerate}[\rm (i)]
\item $|V_1| = |V_2| = \ldots = |V_l|;$
\item $|V_0| < \eps n;$
\item apart from at most  $\eps \binom{l}{2}$ exceptional pairs, the graphs $[V_i, V_j]$ are $(\eps, G)$-regular.
\end{enumerate}

The following version of Szemer\'edi's regularity lemma is well-known. The given prepartition will only be used for the bipartition of the graph $K_{n,n}$ in Lemma~\ref{lem:asymptotic-cycle-partition}~(b). The colours on the edges are represented by the graphs $G_i$.

\begin{lemma}[Regularity lemma with prepartition and colours]
\label{lem:regularity-lemma}
For every $\eps > 0$ and  $m,t \in \mathbb{N}$ there are $M,n_0 \in \mathbb{N}$ such that for all $n \geq n_0$ 
the following holds.\\ For all
graphs $G_0,G_1, G_2, \ldots, G_t$ with $V(G_0) = V(G_1) = \ldots = V(G_t) =V$ and a partition $A_1 \cup \ldots \cup A_s = V$, where
$r \geq 2$ and $|V| = n$, there is a partition $V_0\cup V_1 \cup \ldots \cup V_l$
of $V$ into $l+1$  clusters such that
\begin{enumerate}[\rm (a)]
\item \label{itm:regularity-a} $m \leq l \leq M;$
\item \label{itm:regularity-b} for each $1 \leq i \leq l$ there is a $1 \leq j \leq s$ such that $V_i \subset A_j;$
\item  \label{itm:regularity-c} $V_0\cup V_1 \cup \ldots \cup V_l$ is $(\eps,G_i)$-regular for  each $0\leq i \leq t.$
\end{enumerate}
\end{lemma}

Observe that the regularity lemma provides regularity only for a number of colours bounded by the input parameter $t$. However,
the total number of colours of an $r$-local colouring is not bounded by any function of $r$ (for an example, see Section~\ref{sec:paths}). Luckily, it turns out that it suffices to focus on the
$t$ colours of largest density, where $t$ depends only on $r$ and $\eps$. This is  guaranteed by the following result from~\cite{GLN+87}.

\begin{lemma}
\label{lem:local-edges}
Let a graph $G$ with average degree $a$ be $r$-locally coloured. 
Then one colour has at least  $ a^2/2r^2$ edges.
\end{lemma}

\begin{corollary}
\label{cor:local-edges}
For all $\eps>0$ and $r\in\mathbb N$ there is a $t=t(\eps, r)$ such that 
for any $r$-local colouring of $K_n$ or $K_{n,n}$, there are $t$ colours  such that all but at most $\eps n^2$ edges use these colours.
\end{corollary}
\begin{proof}
We only prove the corollary for $K_{n,n}$, as the proof for $K_n$ is very similar. 
Let $t := \lceil - \frac{2r^2}{\eps}\log \eps\rceil$.
We iteratively take out the edges of the colours with the largest number of edges. We stop either after $t$ steps, or before, if we the remaining graph has density  less than $\eps$. 
At each step  Lemma~\ref{lem:local-edges} ensures that at least a fraction of $\frac{\eps}{2r^2}$ of the remaining edges has the same 
colour.\footnote{Here we use that in a balanced bipartite graph $H$ with $2n$ vertices, $m$ edges, average degree $a$ and density $d$ we have $a^2 = \frac{4m^2}{4n^2} = dm$.}
Hence we can bound the number of edges of the remaining graph by $$\left(1-\frac{\eps}{2r^2}\right)^tn^2\leq e^{-\eps t/2r^2}n^2\leq\eps n^2.$$ 
\end{proof}

\subsubsection{Proof of Lemma~\ref{lem:asymptotic-cycle-partition}}
We only prove part (b) of Lemma~\ref{lem:asymptotic-cycle-partition}, since the proof of part (a) is very similar and actually simpler.
For the sake of readability, we assume that $ n_0 \gg 0$ is sufficiently large and $0 <  \eps \ll 1$ is sufficiently small without calculating exact values.

Let the edges of $K_{n,n}$ with biparts $A_1$ and $A_2$ be coloured $r$-locally and encode the colouring by edge-disjoint graphs $G_1, \ldots, G_s$ on the vertex set of $K_{n,n}$.  By Corollary~\ref{cor:local-edges}, there is a $t= t(\eps,r)$
such that
the union of $G_1, \ldots, G_t$ covers all but at most $\eps n^2 /8r^2$ edges of $K_{n,n}$. We merge the remaining 
edges into $G_{0} := \bigcup_{i=t+1}^s G_i$.
Note that the colouring remains $r$-local and by the choice of $t$, we have
\begin{equation}\label{equ:colour-t+1}
|E(G_{0})| \leq \eps n^2 /8r^2.
\end{equation}

\newcommand{\xknn}{a}
\newcommand{\yknn}{b}

\newcommand{\xr}{v}
\newcommand{\yr}{w}

For  $\eps$, $t$ and $m:=1/\eps$, the regularity lemma (Lemma~\ref{lem:regularity-lemma}) provides $n_0$ and $M$ such for all $n \geq n_0$  there is a vertex-partion 
$V_0,V_1, \ldots, V_{l}$ of $\kn$ satisfying Lemma~\ref{lem:regularity-lemma}(\ref{itm:regularity-a})--(\ref{itm:regularity-c}) 
for $G_0,G_1, \ldots, G_{t}$.

As usual, we define the reduced graph $R$
which has a vertex $\xr_i$ for each cluster $V_i$ for $1 \leq i \leq l$.
We place an edge between vertices $\xr_i $ and $\xr_j$ if the subgraph  $[V_i,V_j]$ of the respective clusters is non-empty and forms
an $(\eps,G_q)$-regular subgraph for all  $0 \leq q  \leq t$. 
Thus, $R$ is a balanced bipartite graph with at least $(1-\eps){\binom{l}{2}}$ edges.

Finally, the colouring of the edges of $K_{n,n}$, induces a \emph{majority colouring} of the edges of $R$. More precisely, we colour
each edge $\xr_i \xr_j$ of $R$  with the colour from $\{0,1,\ldots, t\}$ that appears most on the edges of the subgraph $[V_i,V_j] \subset G$ 
(in case of a tie, pick any of the densest colours).  
Note that if $\xr_i \xr_j$ is coloured~$i$ then by Lemma~\ref{lem:local-edges},
\begin{equation}\label{zweite}
[V_i,V_j]\mbox{ has at least $\frac{1}{2r^2 }(\frac{n}{2l})^2$ edges of colour $i$.} 
\end{equation}

Our next step is to verify that the majority colouring is an $r$-local colouring of $R$. To this end we need the following easy and well-known fact about regular graphs.

\begin{fact}\label{lem:regular-fact}
Let $[A,B]$ be an $\eps$-regular graph of density $d>\eps$. Then at most $\eps |A|$ vertices from $A$ have no neighbours in $B$.
\end{fact}
\begin{claim}\label{cla:regular-r-local}
The colouring of the reduced graph $R$ is $r$-local.
\end{claim}
\begin{proof}
Assume otherwise. Then  there is a vertex $\xr_i \in {V(R)}$ that sees $r+1$ different colours in $R$. By  Fact~\ref{lem:regular-fact}, all but at most $(r+1)\eps |V_i|<|V_i|$ of the vertices in $V_i$ see 
$r+1$ different colours in $K_{n,n}$, contradicting the $r$-locality of our colouring. 
\end{proof}

By~\eqref{equ:colour-t+1}, and by~\eqref{zweite},
we know that $R$ has at most $|E(G_{0})|\frac{4l^2\cdot 2r^2}{n^2} \leq \eps l^2$ edges of colour $0$.
Delete these edges and use Corollary~\ref{cor:robust-cycle-bipartite} to cover all but $o(l)$ vertices of $R$ with $(2r-1)r$  vertex-disjoint 
monochromatic matchings $M^1, \ldots, M^{(2r-1)r}$
of spectrum $1, \ldots, t$. 

We finish by applying  \L uczak's technique for blowing up matching to cycles~\cite{Luc99}. This is done by using the following (by now well-known) lemma.

\begin{lemma}\label{lem_connmatchTOcycles}
Let $t \geq 1$ and $\gamma >0$ be fixed. Suppose   $R$ is the edge-coloured reduced graph of an edge-coloured graph~$H$, for some $\gamma$-regular partition, such that each edge 
$\xr \yr$ of $R$ corresponds to a $\gamma$-regular pair of density at least $\sqrt\gamma$ in the colour of $\xr \yr$.\\
If all but at most $ \gamma |V(R)|$ vertices of $R$ can be covered with $t$ disjoint connected monochromatic matchings, 
then there is a set of at most $t$ monochromatic disjoint cycles in $H$, which together cover all but at most $10\sqrt\gamma|V(H)|$ vertices of $H$.
\end{lemma}

For completeness, let us give an outline of the proof of Lemma~\ref{lem_connmatchTOcycles}.

\begin{proof}[Sketch of a proof of Lemma~\ref{lem_connmatchTOcycles}]
We start by connecting in $H$ the pairs corresponding to matching edges with monochromatic paths of the respective colour, following their connections in $R$. We do this in a cyclic manner, that is, if $v_{i_1}v_{j_1},\dots, v_{i_s}v_{j_s}$
forms the matching, then we take paths $P_1,\dots, P_s$ in a way that
 $P_\ell$ connects $V_{j_\ell}$ and $V_{i_{\ell+1}}$ (modulo $\ell$). The end-vertex of each $P_\ell$  can be taken as a typical vertex of the graph $ [ V_{i_{\ell}}, V_{j_{\ell}}]$ or $[V_{i_{\ell+1}},V_{j_{\ell+1}}]$
 (this is important as we later have to `fill up' the matching edges accordingly). 
We find the connecting paths simultaneously for all matchings.

Note that, as $t$ is fixed,  the paths chosen above together consume only a constant number of vertices of~$H$.
So we can we connect the monochromatic paths using the matching edges, blowing up the edges to long paths, where regularity and density ensure that we can fill up all but a small fraction of the corresponding pairs.
This gives the desired cycles. 

A more detailed explanation of this argument can be found in the proof of the main result of~\cite{GRSS06b}.
\end{proof}

\subsection{The absorbing method}\label{absorbing-method} 
In this subsection we prove Theorem~\ref{thm:complete}. We apply a well known absorbing argument introduced in~\cite{EGP91}. 
To this end we need a few tools.

Call a balanced bipartite subgraph $H$ of a $2n$-vertex graph $\eps$-\emph{hamiltonian}, if any balanced bipartite subgraph of $H$ with at least
$2(1-\eps) n$ vertices  is hamiltonian. The next lemma is a combination of results from~\cite{Hax97,PRR02} and can be found in~\cite{LSS15} in the following explicit form.

\begin{lemma}
\label{lem:eps-hamiltonian}
For any $1 >\gamma > 0,$ there is an $n_0 \in
\mathbb{N}$ such that any
balanced bipartite graph on $2n \geq 2n_0$ vertices and of
edge density at least $\gamma$ has a $\gamma/4$-hamiltonian subgraph
of size at least $\gamma^{3024/\gamma} n/3$.
\end{lemma}

The following lemma is  taken from~\cite{CS16}.

\begin{lemma}\label{lem:absorb}
Suppose that $A$ and $B$ are vertex sets with $|B| \leq |A|/r^{r+3}$ and the edges of the 
complete bipartite graph between $A$ and $B$ are  $r$-locally coloured. Then all vertices 
of $B$ can be covered with at most $r^2$ disjoint monochromatic cycles.
\end{lemma}

\begin{proof}[Sketch of a proof of Theorem~\ref{thm:complete}]
Here we only prove part (b) of Theorem~\ref{thm:complete}, since the proof of (a) is almost identical. The differences are discussed at the end of the section.

Let $A$ and $B$ be the two partition classes of the $r$-locally edge coloured $K_{n,n}$.
We assume that $n \ge n_0$, where we specify $n_0$ later.
Pick subsets $A_1 \subseteq A$ and $B_1 \subseteq B$ of size $\lceil n / 2 \rceil$ each.
Say red is the majority colour of $[A_1,B_1]$. Then by Lemma~\ref{lem:local-edges}, there are at least
$n^2/8r^2$ 
red edges in $[A_1,B_1]$. 

Lemma~\ref{lem:eps-hamiltonian} applied with $\gamma = 1/10r^2$ yields a red $\gamma/4$-hamiltonian sub\-graph $[A_2,B_2]$ of $[A_1,B_1]$ with 
\[
|A_2|=|B_2|\ge  \gamma^{3024/\gamma} |A_1|/3  \ge \gamma^{3024/\gamma} n/7 .
\]
Set $H := G - (A_2 \cup B_2)$, and note that each bipart of $H$ has order at least $\lfloor n / 2 \rfloor$.
Let $\delta := \gamma^{4000/\gamma}$.
Assuming $n_0$ is large enough, Lemma~\ref{lem:asymptotic-cycle-partition}(b) provides $(2r-1)r$ monochromatic vertex-disjoint cycles  covering all but at most
$2 \delta n$ vertices of~$H$.
Let $X_A \subseteq A$ (resp.~$X_B \subseteq B$) be the set of uncovered vertices in $A$ (resp.~$B$).
Since we may assume none of the monochromatic cycles is an isolated vertex, we have $|X_A|=|X_B| \le \delta n$.

By the choice of $\delta$, and since we assume $n_0$ to be sufficiently large, we can apply Lemma~\ref{lem:absorb}
to the bipartite graphs $[A_2,X_B]$ and $[B_2,X_A]$.
This gives~$2r^2$ vertex-disjoint monochromatic cycles that together cover $X_A \cup X_B$.  
Again, we assume none of these cycles is trivial.
As $|X_A|=|X_B| \le \delta n$, we know that the remainder of $[A_2,B_2]$ contains
a red Hamilton cycle.
Thus, in total, we found a cover of $G$ with at most $(2r-1)r+2r^2+1 \leq 4 r^2$
vertex-disjoint monochromatic cycles.

As claimed above, the proof of Theorem~\ref{thm:complete}(a) is very similar. 
The main difference is that instead of an $\eps$-hamiltionan subgraph we use a large red 
triangle cycle. A triangle cycle $T_k$ consists of a cycle on $k$ vertices $\{v_1, \ldots , v_k\}$ and $k$ additional vertices $A = \{a_1, \ldots a_k\}$, where
$a_i$ is joined to $v_i$ and $v_{i+1}$ (modulo~$k$). Note that $T_k$ remains hamiltionan after the deletion of any subset of vertices of~$A$. We use some classic Ramsey theory
to find a large monochromatic triangle cycle $T_k$ in an $r$-locally coloured $K_n$, as
shown in~\cite{CS16}. Next, Lemma~\ref{lem:asymptotic-cycle-partition}(a) guarantees we can cover most vertices of $K_n \setminus T_k$ with  $r(r+1)/2$ monochromatic cycles. We finish by absorbing
the remaining vertices $B$ into $A$ with only one application of Lemma~\ref{lem:absorb}, thus producing $r^2$ additional cycles. As noted above, the remaining part of $T_k$ is hamiltionan and so
we have partitioned $K_n$ into $r(r+1)/2 + r^2 +1 \leq 2r^2$ monochromatic cycles.
\end{proof}

\section{Bipartite graphs with 2-local colourings}
\label{sec:2-local-path-partition}

In this section we prove Theorem~\ref{thm:bipartite-2-local-paths-partition} and Theorem~\ref{thm:ramsey}.
We start by specifying the structure of 2-local colourings of $K_{n,n}$.
Let $G$ be any graph, and let the edges of $G$ be coloured arbitrarily with  colours in $\mathbb{N}$.
We denote by $C_i$ the subgraph of $G$ induced by vertices that are adjacent to any edge of colour $i$.
Note that $C_i$ can contain edges of colours other than $i$.
If for colours $i,j$ the intersection $V(C_i) \cap V(C_j)$ is empty, we can merge $i$ and $j$ as we are only interested in monochromatic paths. 
We call an edge colouring \emph{simple}, if $V(C_i) \cap V(C_j) \neq \emptyset$ for all colours $i,j$ that appear on an edge.

In~\cite{GLS+87} it was shown that the number of colours in a simple 2-local colouring of $K_n$ is bounded by 3. 
In the next lemma we will see that for $K_{n,n}$ the number of colours in a simple 2-local colouring is bounded by 4. 
For $r \geq 3$, however, simple $r$-local colourings can have an arbitrary large number of colours: 
take a $t \times t$ grid $G$ and colour the edges of the column $i$ and row $i$ with colour $i$ for $1\leq i \leq t$. Then add edges of a new colour $t+1$ until $G$ is complete (or complete bipartite) and observe that $G$ is 3-locally edge coloured and simple, but the total number of colours is $t+1$.

In what follows, we denote partition classes of a bipartite graph $H$ (which we imagine as either top and bottom) by $\bitop{H}$ and $\bibot{H}$.
\begin{lemma}\label{3possiblecolourings}
Let $K_{n,n}$ have a simple 2-local colouring. Then the total number of colours is at most four. In particular, if there are (edges of) colours 1,2,3 and~4, then
\begin{itemize}
\item $\bitop{\kn}= \bitop{C_1  \cap C_2 } \cup \bitop{C_3  \cap C_4 }$ and	
\item $\bibot{\kn}= \bibot{C_1 \cap C_3} \cup \bibot{C_1 \cap C_4} \cup \bibot{C_2  \cap C_3 } \cup \bibot{C_2  \cap C_4 }$
\end{itemize}
as shown in Figure~\ref{fig:c} (modulo swapping colours and swapping $\bitop\kn$ with $\bibot\kn$).
\end{lemma}

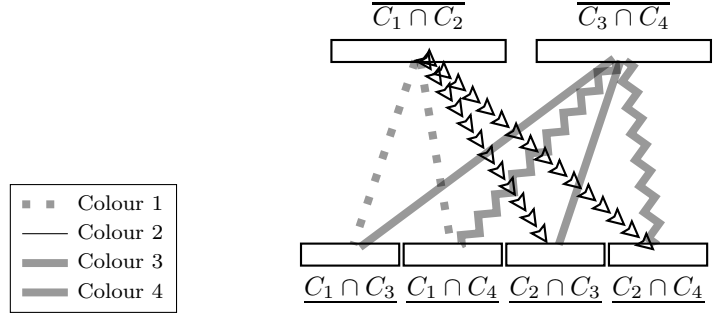
\begin{figure}
\captionsetup[subfigure]{labelformat=empty}
\centering
\begin{subfigure}[b]{0.3\textwidth}
\begin{tikzpicture}
\begin{customlegend}[legend cell align=left, 
legend entries={ 
	Colour 1, Colour 2, Colour 3, Colour 4
},
legend style={at={(0,0)},font=\footnotesize}] 
\addlegendimage{draw,  opacity=0.4,line width=3, loosely dashed}
\addlegendimage{draw, decorate,decoration={shape backgrounds,shape=dart,shape size=2mm}}
\addlegendimage{draw,  opacity=0.4,line width=3}
\addlegendimage{draw,  opacity=0.4,line width=3, decorate, decoration=zigzag}
\end{customlegend}
\end{tikzpicture}
\end{subfigure} 
\begin{subfigure}[b]{0.3\textwidth}
\begin{tikzpicture}[thick,scale=0.9,
every node/.style={}, tedge/.style={opacity=0.4,line width=3},
gfit/.style={rectangle,draw,inner sep=0pt,text width=0.0cm},
rfit/.style={rectangle,draw,inner sep=0pt,text width=0.0cm}
]

\foreach \i in {1,2,...,6}  \node (t\i) at (\i,3) {};
\foreach \i in {1,2,...,12}  \node (b\i) at (0.5*\i,0) {};

\node [gfit,fit=(t1) (t3),label=above:{$\bitop{C_1 \cap C_2}$}] {};
\node [gfit,fit=(t4) (t6),label=above:{$\bitop{C_3 \cap C_4}$}] {};

\node [gfit,fit=(b1) (b3),label=below:{$\bibot{C_1 \cap C_3}$}] {};
\node [gfit,fit=(b4) (b6),label=below:{$\bibot{C_1 \cap C_4}$}] {};
\node [gfit,fit=(b7) (b9),label=below:{$\bibot{C_2 \cap C_3}$}] {};
\node [gfit,fit=(b10) (b12),label=below:{$\bibot{C_2 \cap C_4}$}] {};

%
\foreach \b in {2,5}
\path[draw, tedge,loosely dashed] (t2) -- (b\b);

\foreach \b in {2,8}
\path[draw,  tedge]  (t5) -- (b\b);

\foreach \b in {5,11}
\path[draw,  tedge, decorate, decoration=zigzag](t5) -- (b\b);

\foreach \b in {8,11}
\path[draw, decorate,decoration={shape backgrounds,shape=dart,shape size=2mm}] (t2) -- (b\b);

\end{tikzpicture}
\end{subfigure}
\caption{The four colour case of Lemma~\ref{3possiblecolourings}.}
\label{fig:c}
\end{figure}

\begin{proof}[Proof of Lemma~\ref{3possiblecolourings}]
We can assume there are at least four colors  in total, as otherwise there is nothing to show.
We start by observing that for any four distinct colours $i,j,k,\ell$, if $v\in V(C_i\cap C_j)$ and $w\in V(C_k\cap C_\ell)$, then, by 2-locality, $v$ and $w$ cannot lie in opposite classes of $K_{n,n}$. Thus either $V(C_i\cap C_j)\cup V(C_k\cap C_\ell)\subseteq \bibot{K_{n,n}}$ or $V(C_i\cap C_j)\cup V(C_k\cap C_\ell)\subseteq \bitop{K_{n,n}}$. Fixing four colours $1,2,3,4$, and considering their six (by simplicity non-empty) intersections, the pigeon-hole principle gives that  (after possibly swapping colours and/or top and bottom class of $K_{n,n}$),
\begin{equation}\label{equ:blabla}
	V(C_1 \cap C_3)\cup V(C_2 \cap C_4)\cup V(C_1  \cap C_4)\cup V(C_2  \cap C_3)\subseteq \bibot{K_{n,n}}.
\end{equation}
As every colour must see both top and bottom of $K_{n,n}$, 
we have that $V(C_1 \cap C_2)\cup V(C_3 \cap C_4)\subseteq \bitop{K_{n,n}}.$ By 2-locality there are no other colours.
\end{proof}

\subsection{Partitioning into paths}\label{sec:paths}
In this subsection we prove Theorem~\ref{thm:bipartite-2-local-paths-partition}. For the sake of contradiction,  assume that $K_{n,n}$ is 2-locally edge-coloured such that there is no partition into three monochromatic paths.   
Since we are not interested in the actual colours of the path we can assume the colouring to be simple.   
Furthermore Theorem~\ref{thm:pokrovskiy} implies that there are at least three colours.   

A path is \emph{even} if it has an even number of vertices.
\begin{claim}\label{onethenthree}
There is no even monochromatic path  $P$ such that $\bitop{\kn \setminus P}$ is contained in $\bitop{C_i \cap C_j}$ for distinct colours $i,j$. 
\end{claim}
\begin{proof}
Suppose the contrary and let $P$ be as described in the claim and of maximum length. Since the colouring is $2$-local and $\bitop{\kn \setminus P} \subset \bitop{C_i \cap C_j}$, the graph on $\kn \setminus P$ is 2-coloured. Using Theorem~\ref{thm:pokrovskiy}, we are done unless the colouring on ${\kn \setminus P}$ is split.

In that case, let $p$ be the endpoint of $P$ in $\bibot{\kn}$. Since $\bitop{\kn \setminus P} \subset \bitop{C_i \cap C_j}$, the edges between $p$ and $\bitop{\kn \setminus P}$ have colours $i$ or $j$. So $P$ has colour $k\notin\{i,j\}$, as otherwise we could use the splitness of ${\kn \setminus P}$ to extend $P$ with two extra vertices. But then, $p$ can only see one more colour apart from $k$, so we may assume that all the edges between $p$ and $\bitop{\kn \setminus P}$ have colour $i$. Now cover $\kn \setminus P$ by two paths $P_1$ and $P_2$ of the colour $i$ and one path of the colour $j$. The paths $P_1$ and $P_2$ can be joined using the vertex $p$ to give the three required paths.
\end{proof}

Now the case of four colours of Lemma~\ref{3possiblecolourings} is easily solved:   
without loss of generality suppose that $|\bitop{C_1 \cap C_2}| \leq n/2$.   
By symmetry between colours $1$ and~$2$ we can assume that $|\bitop{C_2}| \leq |\bibot{C_2}|$.   
So there exists an even colour 2 path $P$ covering $\bitop{C_2}=\bitop{C_1\cap C_2}$ and we are done by Claim~\ref{onethenthree}. This proves the following claim.
\begin{claim}\label{drei}
The total number of colours is three.
\end{claim}

Our next aim is to show that the colouring looks like in Figure~\ref{fig:b}, that is, that every vertex sees two colours. For this, we need the next claim and the following definition. 
We say that a subgraph of $H \subset K_{n,n}$ is connected in colour~$i$, if every two vertices of $H$ are connected by a path  of colour $i$ in $H$.

\begin{claim}\label{greatobservation}
There is no even monochromatic path $P$ such that $\kn \setminus P$ is connected in some colour $i$.
\end{claim}
\begin{proof}
Assume the opposite and let $P$ be as described in the claim. Simplify the colouring of $\kn\setminus V(P)$ to a $2$-colouring by merging all colours distinct from $i$. (Note that since all vertices see $i$, by 2-locality no vertex can see more than one of the merged colours.)  The new colouring is not a split colouring by the assumption on $i$. Hence Theorem~\ref{thm:pokrovskiy} applies, and we are done.
\end{proof}

\begin{claim}\label{aI}
Each vertex sees two colours.
\end{claim}
\begin{proof}
Suppose that there is a vertex in $\bitop{K_{n,n}}$ that sees only colour, $1$ say. 
Then by 2-locality $\bibot{C_2 \cap C_3} = \emptyset$.
Since the colouring is simple we know that $\bitop{C_2 \cap C_3} \neq \emptyset$.  
Therefore $\bibot{K_{n,n} }\subset \bibot{(C_1 \cap C_2) \cup (C_1 \cap C_3)}$.		
If $|\bitop{C_2 \cap C_3}| > |\bibot{C_1 \cap C_3}|$, we can choose an even path of colour 3  that contains all vertices of $\bibot{C_1 \cap C_3}$ and apply Claim~\ref{onethenthree}.
Otherwise, let $P$ be an even  path of colour 3 between $|\bitop{C_2 \cap C_3}| $ and $ |\bibot{C_1 \cap C_3}|$ that covers all vertices of $\bitop{C_2 \cap C_3}$. 
Since all remaining vertices lie in $C_1$, the subgraph $K_{n,n} \setminus P$ is connected in colour 1  and we are done by Claim~\ref{greatobservation}.
\end{proof}

\begin{figure}
	\captionsetup[subfigure]{labelformat=empty}
	\centering
	\begin{subfigure}[b]{0.3\textwidth}
		\begin{tikzpicture}
		\begin{customlegend}[legend cell align=left, 
		legend entries={ 
			Colour 1, Colour 2, Colour 3
		},
		legend style={at={(0,0)},font=\footnotesize}] 
		\addlegendimage{draw,  opacity=0.4,line width=3, loosely dashed}
		\addlegendimage{draw, decorate,decoration={shape backgrounds,shape=dart,shape size=2mm}}
		\addlegendimage{draw,  opacity=0.4,line width=3}
		\addlegendimage{draw,  opacity=0.4,line width=3}
		\end{customlegend}
		\end{tikzpicture}
	\end{subfigure}  
	\begin{subfigure}[b]{0.3\textwidth}
	\begin{tikzpicture}[thick,scale=0.9,
	every node/.style={}, tedge/.style={opacity=0.4,line width=3},
	gfit/.style={rectangle,draw,inner sep=0pt,text width=0.0cm},
	rfit/.style={rectangle,draw,inner sep=0pt,text width=0.0cm}
	]

	\foreach \i in {1,2,...,9}  \node (t\i) at (0.5*\i,3) {};
	\foreach \i in {1,2,...,9}  \node (b\i) at (0.5*\i,0) {};
	
	\node [gfit,fit=(t1) (t3),label=above:{$\bitop{C_1 \cap C_2}$}] {};
	\node [gfit,fit=(t4) (t6),label=above:{$\bitop{C_1 \cap C_3}$}] {};
	\node [gfit,fit=(t7) (t9),label=above:{$\bitop{C_2 \cap C_3}$}] {};

	\node [gfit,fit=(b1) (b3),label=below:{$\bibot{C_1 \cap C_2}$}] {};
	\node [gfit,fit=(b4) (b6),label=below:{$\bibot{C_1 \cap C_3}$}] {};
	\node [gfit,fit=(b7) (b9),label=below:{$\bibot{C_2 \cap C_3}$}] {};
	
	\foreach \b in {1,4} \foreach \t in {1,4}
	\path[draw, tedge,loosely dashed] (t\t) -- (b\b);

	\foreach \b in {2,8} \foreach \t in {2,8}
	\path[draw, decorate,decoration={shape backgrounds,shape=dart,shape size=2mm}] (t\t) -- (b\b);
	
	\foreach \b in {6,9} \foreach \t in {6,9}
	\path[draw, tedge] (t\t) -- (b\b);
	\end{tikzpicture}
	\end{subfigure} 	\qquad \qquad
	\caption{There are three colours and each vertex sees exactly two colours.}
	\label{fig:b}
\end{figure}
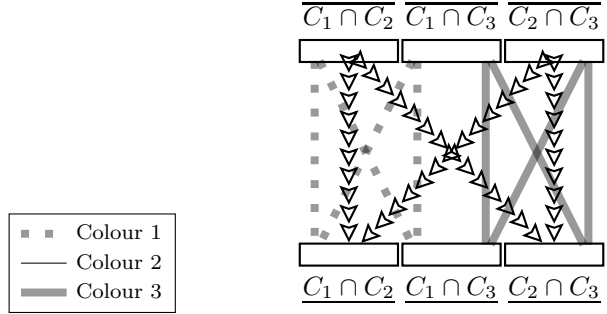

Claims~\ref{drei} and~\ref{aI} ensure that for the rest of the proof we can assume that the colouring is exactly as shown in Figure~\ref{fig:b} (with some of the sets possibly being empty). Now, let us see how Claim~\ref{onethenthree} implies that we easily find the three paths if one of the $C_i$ is complete bipartite in colour $i$.

\begin{claim}\label{nocompletebipcomp}
 For $i \in \{1,2,3\}$, the graph $C_i$ is not complete bipartite in colour $i$.
\end{claim}
\begin{proof}
Suppose the contrary and let $C_i$ contain only edges of colour $i$. 
Take out a longest even path of colour $i$ in $C_i$.
This leaves us either with only $\bibot{C_j \cap C_k}$ in the bottom partition class, or with only $\bitop{ C_j \cap C_k}$ in the top partition class (where~$j$ and $k$ are the other two colors). 
We may thus finish by applying Claim~\ref{onethenthree}, after possibly switching top and bottom parts. 
\end{proof}

\begin{claim}\label{aII}\label{all} 
 For $i \in \{1,2,3\}$, the graph $C_i$ is connected in colour $i$.
\end{claim}
\begin{proof}
	For contradiction, suppose that $C_3$ is not connected in colour 3 (the other colours are symmetric). Then there are two edges $e,f$ of colour $3$ belonging to  $C_3$ that are not joined by a path of colour $3$. First assume we can choose $e$ in $E(C_2 \cap C_3)$. Since all edges between  $C_1 \cap C_3$ and $C_2 \cap C_3$ have colour 3, we get $f\in E(C_2 \cap C_3)$, and $C_1 \cap C_3$ has no vertices. But this contradicts our assumption that the colouring is simple. Therefore, $C_2 \cap C_3$ and, by symmetry, $C_1 \cap C_3$ contain no edges of colour~3.
	
	By symmetry (between the top and bottom partition) we can assume that $|\bitop{C_1 \cap C_3}| \geq |\bibot{C_1 \cap C_3}|$.
	Further, we have $|\bitop{C_1 \cap C_3}| < |\bibot{C_1 \cap C_2}| + |\bibot{C_1 \cap C_3}|$, since otherwise we could find an even  path of colour 1 that covers all of $\bibot{C_1 \cap C_2} \cup \bibot{C_1 \cap C_3}$ and use Claim~\ref{onethenthree}. 
	So we can choose an even path $P$ of colour 1,  alternating between $\bitop{C_1 \cap C_3}$ and  $\bibot{C_1 \cap C_2} \cup \bibot{C_1 \cap C_3}$, that contains both $\bitop{C_1 \cap C_3}$ and $\bibot{C_1 \cap C_3}$. 
	Thus $K_{n,n} \setminus P$ is connected in colour 2 and Claim~\ref{greatobservation} applies.
\end{proof}

Let us now show that for pairwise distinct $i,j,k \in \{1,2,3\}$ we have
\begin{equation}\label{equ:maya-1}
\text{at least one of $\bibot{C_i\cap C_j}, \bitop{C_i\cap C_k}$ is not empty.}
\end{equation}
To see this,  note that the edges between $\bitop{C_i \cap C_j}$ and $\bibot{C_i \cap C_k}$ are of colour $i$. 
Thus if~\eqref{equ:maya-1} does not hold, we can find a colour $i$ (possibly trivial) path~$P$ that covers one of these two sets. 
Hence either in the top or in the bottom part of $K_{n,n}$, the path $P$ covers all but $C_j\cap C_k$.
We can thus finish with Claim~\ref{onethenthree}.

Together with the fact that every colour must see both top and bottom class,~\eqref{equ:maya-1} immediately implies that for pairwise distinct $i,j,k \in \{1,2,3\}$ we have
\begin{equation}\label{equ:maya-xx}
\text{at least one of $C_i\cap C_j, C_i\cap C_k$ meets both $\bibot{K_{n,n}}$ and $\bitop{K_{n,n}}$.}
\end{equation}

So, of the three bipartite graphs $C_i\cap C_j$, two have non-empty tops and bottoms. Hence, after possibly swapping colours, we know that
 the four sets $\bibot{C_1\cap C_i}$, $\bitop{C_1\cap C_i}$, $i = 2,3$, are non-empty.
Observe that after possibly swapping colours $2$ and $3$, and/or switching partition classes of $\kn$, we have one of the following situations:
\begin{enumerate}[(i)]
\item $|\bitop{C_1\cap C_2}| \geq |\bibot{C_1\cap C_3}|$ and $|\bibot{C_1\cap C_2}| \geq |\bitop{C_1\cap C_3}|$, or
\item $|\bitop{C_1\cap C_2}| \geq |\bibot{C_1\cap C_3}|$ and $|\bibot{C_1\cap C_2}| \leq |\bitop{C_1\cap C_3}|$.
\end{enumerate}

In either of these situations, note that as all involved sets are non-empty, by Claim~\ref{all}
there is an edge $e_1$ of  colour $1$ in $E(C_1\cap C_2)\cup E(C_1 \cap C_3)$.
So if we are in situation~(ii), we can find an even  path of colour 1 covering all of $\bibot{C_1\cap C_3}\cup\bibot{C_1\cap C_2}$. Now Claim~\ref{onethenthree} applies, and we are done. So assume from now on we are in situation~(i).

 Similarly as above, by~\eqref{equ:maya-xx}, there is an edge $e_2$ of colour 2 in $E(C_3\cap C_2)\cup E(C_1\cap C_2)$. By Claim~\ref{nocompletebipcomp},  $C_3$ is not complete bipartite in colour 3. So we can assume that at least one of $e_1$ or $e_2$ is chosen in~$C_3$ and hence the two edges are not incident.

Extend $e_1$  to an even colour 1 path $P$ covering all of $C_1\cap C_3$, using (apart from $e_1$) only edges from
$[\bibot {C_1\cap C_3},\bitop {C_1\cap C_2}]$ and from $[\bibot {C_1\cap C_2},\bitop {C_1\cap C_3}]$, while avoiding the endvertices of $e_2$, if possible.
If we had to use one of the endvertices of $e_2$ in $P$, then $P$ either covers 
all of $\bibot{C_1\cap C_2}$ or all of $\bitop{C_1\cap C_2}$. In either case we may apply Claim~\ref{onethenthree}, and are done. 
On the other hand, if we could avoid both endvertices of $e_2$ for $P$, then Claim~\ref{greatobservation} applies and we are done.
This finishes the proof of Theorem~\ref{thm:bipartite-2-local-paths-partition}.

\subsection{Finding long paths}

In this subsection we prove Theorem~\ref{thm:ramsey}. We will use the following theorem, which resolves the problem for the case of of 2-colourings.

\begin{theorem}[\cite{FS75,GL73}]\label{thm:faudree-schelp}
Every $2$-edge-coloured $K_{p+q-1,p+q-1}$ contains a colour~1 path of length $2p$ or a  colour 2 path of length $2q$. 
\end{theorem}

 As in the last section, $C_i$ denotes the subgraph induced by the vertices that have an edge of colour $i$. Recall that the length of a path is the number of its vertices.

\begin{lemma}\label{lem:intersection-cover-path}
	Let $K_{2m-1,2m-1}$ be  2-locally coloured with colours $1,2,3$. Then for distinct colours $i,j$ there is a monochromatic path of length at least $$\min\{2m,2\max(|\bitop{C_i \cap C_j}|,|\bibot{C_i \cap C_j}|)\}.$$
\end{lemma}
\begin{proof}
	By symmetry, we can assume that $|\bitop{C_i \cap C_j}| \geq |\bibot{C_i \cap C_j}|$. Moreover, we can assume that $\bitop{C_i \cap C_j}\neq\emptyset$, as otherwise there is nothing to prove. Then by 2-locality, 
	\begin{equation}\label{kkk}
		\bibot{C_k \setminus (C_i \cup C_j)} = \emptyset,
	\end{equation}
	where $k$ denotes the third colour.
	
	We apply Theorem~\ref{thm:faudree-schelp} to a balanced subgraph of $C_i \cap C_j$ with $p = m - |\bibot{C_i \setminus C_j}|$ and $q = m - |\bibot{C_j \setminus C_i}|$. 	
	For this, note that we have $$p+q-1  = 2m-1-|\bibot{C_i \setminus C_j}| -|\bibot{C_j \setminus C_i}|  \overset{\eqref{kkk}}{=} |\bibot{C_i \cap C_j}|\leq |\bitop{C_i \cap C_j}|.$$
	By symmetry between $i$ and $j$ we can assume that the outcome of Theorem~\ref{thm:faudree-schelp} is a colour $i$ path $P$ of length $2(m - |\bibot{C_i \setminus C_j}|)$.
	Let $R \subset [\bitop{C_i \cap C_j}\setminus \bitop{P}, \bibot{C_i \setminus C_j}]$ be a path of colour $i$ and length $$r = \min(2|\bitop{C_i \cap C_j} \setminus \bitop{P}|,2|\bibot{C_i \setminus C_j}|).$$
	If $r = 2|\bibot{C_i \setminus C_j}|$, then we can join $P$ and $R$ to a path of length of $2m$.
	Otherwise $r = 2|\bitop{C_i \cap C_j} \setminus \bitop{P}|$ and we can join $P$ and $R$ to a path of length of $2|\bitop{C_i \cap C_j}|$.
\end{proof}

Now let us prove Theorem~\ref{thm:ramsey} by contradiction. To this end, assume that $K_{2\m-1,2\m-1}$ is coloured 2-locally and has no monochromatic path on $2\m$ vertices. Since we are not interested in the actual colours of the path we can assume the colouring to be simple, as in the previous subsection. Furthermore Theorem~\ref{thm:faudree-schelp} implies that there are at least three colours. 

We now apply Lemma~\ref{3possiblecolourings}. 
The four colour case of Lemma~\ref{3possiblecolourings} is quickly resolved: Without loss of generality suppose that $|\bitop{C_1 \cap C_2}| \geq \m$. By symmetry between colours $1$ and $2$, we can assume that $|\bibot{C_1 \cap C_3} \cup \bibot{C_1 \cap C_4}| \geq \m$. Thus we easily find a colour 1 path of length $2m$ alternating between these sets. This proves:

\begin{claim}\label{drei2}
	The total number of colours is three.
\end{claim}

We can now exclude vertices that see only one colour.

\begin{claim}\label{only-two-colours}
	Each vertex sees two colours.
\end{claim}
\begin{proof}
	Suppose that there is a vertex in $\bitop{K_{2\m-1,2\m-1}}$ that sees only colour $1$, say. 
	Then by 2-locality, $\bibot{C_2 \cap C_3} = \emptyset$.
	Since the colouring is simple we know that $\bitop{C_2 \cap C_3} \neq \emptyset$.  
	Therefore $\bibot{K_{2\m -1,2\m -1} }\subset \bibot{(C_1 \cap C_2) \cup (C_1 \cap C_3)}$.		
Since one of $\bibot{C_1 \cap C_2}$ and $\bibot{C_1 \cap C_3}$ must have size at least $\m$, we are done by Lemma~\ref{lem:intersection-cover-path}.
\end{proof} 

Put together, Claims~\ref{drei2} and~\ref{only-two-colours} allow us to assume that the colouring is as shown in Figure~\ref{fig:b}. The next claim follows instantly from Lemma~\ref{lem:intersection-cover-path}.

\begin{claim}\label{cla:eine-seite-kleiner-n/2}
For distinct colours $i, j$ we have $\max(|\bitop{C_i \cap C_j}|,|\bibot{C_i \cap C_j}|) < \m.$
\end{claim}
As the three top parts sum up to $2m-1$, and so do the three bottom parts, we immediately get:

\begin{claim}\label{cla:non-empty-parts}
$\bibot{C_i \cap C_j},\bitop{C_i \cap C_j}\neq \emptyset$ for all distinct $i,j \in \{1,2,3\}$.
\end{claim}

The next claim requires some more work. Recall that a subgraph of $H \subset K_{n,n}$ is connected in colour~$i$, if every two vertices of $H$ are connected by a path  of colour $i$ in $H$.

\begin{claim}\label{cla:zwei-seiten-kleiner-n/2}
If the subgraph $C_i$ is connected in colour $i$, then there are distinct $j,k \in \{1,2,3\} \setminus \{i\}$ such that $|\bitop{C_i \cap C_j}| \geq |\bibot{C_i \cap C_k}|$, $|\bibot{C_i \cap C_j}| > |\bitop{C_i \cap C_k}|$ (modulo swapping top and bottom partition classes) and $|V(C_i \cap C_k)|< \m$.
\end{claim}

\begin{proof}
Suppose that $C_i$ is connected in colour $i$ and let $j,k \in \{1,2,3\} \setminus \{i\}$ be such that $|\bitop{C_i \cap C_j}| \geq |\bibot{C_i \cap C_k}|$ (after possible swapping top and bottom partition). By  Claim~\ref{cla:non-empty-parts}, and as $C_i$ is connected in colour $i$, we find an edge $e_i \in E(C_i \cap C_j) \cup E(C_i \cap C_k)$ of colour $i$.
 Choose  an even path $P \subset [\bitop{C_i \cap C_j}, \bibot{C_i \cap C_k}]$ which covers $\bibot{C_i \cap C_k}$ and  ends in one of the vertices of $e_i$.

For the first part of the claim, assume to the contrary that  $|\bibot{C_i \cap C_j}| \leq |\bitop{C_i \cap C_k}|$. 
Take an even path $P' \subset [\bibot{C_i \cap C_j}, \bitop{C_i \cap C_k}]$ which covers $\bibot{C_i \cap C_j}$  and ends in a vertex of $e_i$.
Since $P$ and $P'$ are joined by $e_i$ we infer that $|\bibot{C_i \cap C_k}|+ |\bibot{C_i \cap C_j}| < \m$. 
But then $|\bibot{C_j \cap C_k}| \geq \m$ in contradiction to Claim~\ref{cla:eine-seite-kleiner-n/2}.
This shows that $|\bibot{C_i \cap C_j}| > |\bitop{C_i \cap C_k}|$, as desired.

This allows us to pick an even path $P'' \subset [\bibot{C_i \cap C_j}, \bitop{C_i \cap C_k}]$ of colour $i$, which covers $\bitop{C_i \cap C_k}$ and  ends in one of the vertices of $e_i$.
Join $P$ and $P''$ via $e_i$ to obtain a colour $i$ path of length at least $2 |\bitop{C_i \cap C_k}| + 2 |\bibot{C_i \cap C_k}| =2|V(C_i \cap C_k)|$. So by our assumption that there is no monochromatic path of length $2\m$, we obtain $|V(C_i \cap C_k)|<\m$, as desired.
\end{proof}

\begin{claim}\label{cla:hoechstens-eine-kleiner-n/2}
For at most one pair of distinct indices $i,j \in\{1,2,3\}$ it holds that $|V(C_i \cap C_j)|<\m$.
\end{claim}
\begin{proof}
Suppose, on the contrary, that $C_1 \cap C_2$ and $C_1 \cap C_3$ each have  less than $\m$ vertices. 
Then $C_2 \cap C_3$ has at least $2\m$ vertices. 
Therefore one of its partition classes has size at least $\m$, a contradiction to Claim~\ref{cla:eine-seite-kleiner-n/2}. 
\end{proof}

We are now ready for the last step of the proof of  Theorem~\ref{thm:ramsey}.
We start by observing that if for some $i \in \{1,2,3\}$, the subgraph $C_i$ is not connected in colour $i$, then (letting $j,k$ be the other two indices) the edges of the  graphs $C_i \cap C_j$ and $C_i \cap C_k$ are all of colour $j$, or colour $k$, respectively, and thus both $C_j$ and $C_k$ are connected in colour $j$, or colour $k$, respectively. So we can assume that there are at least two distinct indices $j,k \in \{1,2,3\}$, such that the subgraphs $C_j$, $C_k$ are  connected in colour $j$, or in colour $k$, respectively. Say these indices are $1$ and $3$.

We use Claim~\ref{cla:zwei-seiten-kleiner-n/2} twice: 
For $C_1$ it yields that one of $C_1 \cap C_3$ and $C_1 \cap C_2$ has  less than $\m$ vertices.
For $C_3$ it yields that one of $C_1 \cap C_3$ and $C_2 \cap C_3$ has  less than $\m$ vertices.
So by Claim~\ref{cla:hoechstens-eine-kleiner-n/2} we get that necessarily, 
\begin{equation}\label{xxxxx}
|V(C_1 \cap C_3)|<\m, \ |V(C_1 \cap C_2)|\geq \m, \ |V(C_2 \cap C_3)|\geq \m.
\end{equation}
 Again using Claim~\ref{cla:zwei-seiten-kleiner-n/2}, this implies that $C_2$ is not connected in colour $2$. So by Claim~\ref{cla:non-empty-parts} and the fact that the edges between $C_1 \cap C_2$ and $C_2 \cap C_3$ are complete bipartite in colour 2, we have that  \begin{equation}\label{xxxxx+}\text{$C_1 \cap C_2$ is complete bipartite in colour 1.}
 \end{equation}
Also, in  light of~\eqref{xxxxx}, Claim~\ref{cla:zwei-seiten-kleiner-n/2} with input $i=1$ gives $j=2$ and $k = 3$ and thus $|\bitop{C_1 \cap C_2}| \geq |\bibot{C_1 \cap C_3}|$, $|\bibot{C_1 \cap C_2}| > |\bitop{C_1 \cap C_3}|$ (after possibly swapping top and bottom partition).
Choose two balanced paths of colour 1:
The first path $P \subset [\bitop{C_1 \cap C_2}, \bibot{C_1 \cap C_3}]$ such that it covers $\bibot{C_1 \cap C_3}$.
The second path $P' \subset [\bibot{C_1 \cap C_2}, \bitop{C_1 \cap C_3}]$ such that it covers $\bitop{C_1 \cap C_3}$.
As by~\eqref{xxxxx+} we know that $C_1 \cap C_2$ is complete bipartite in colour 1, we can join $P$ and $P'$ with a  path of colour 1 in $C_1 \cap C_2$, such that the resulting path $P''$ covers one of  $\bitop{C_1}$, $\bibot{C_1}$.
Since by assumption, $P''$ has less than~$2 \m$ vertices, we obtain that $\bitop{C_2 \cap C_3}$ or $\bibot{C_2 \cap C_3}$ has  size  at least $\m$, a contradiction to Claim~\ref{cla:eine-seite-kleiner-n/2}.
This finishes the proof of Theorem~\ref{thm:ramsey}.

\bibliographystyle{amsplain}
\bibliography{monochromatic-cover}

\end{document}